\documentclass[aop,preprint]{imsart} 
\usepackage{amsthm,amsmath,amsfonts,amssymb}
\usepackage{mathdots}
\usepackage{graphicx}
\usepackage{bm}
\usepackage{mathtools}
\usepackage{mathrsfs}
\usepackage{type1cm}
\usepackage{latexsym,enumitem}
\usepackage{dsfont}
\usepackage{bbm}
\usepackage{soul}
\usepackage{epic,tikz}

\usepackage{tikz}

\usepackage[
  a4paper,
  left=30mm,
  right=30mm,
  top=25mm,
  bottom=30mm
]{geometry}
\usetikzlibrary{matrix}

\renewcommand {\epsilon}{\varepsilon}
\renewcommand {\le}{\leqslant}
\renewcommand {\ge}{\geqslant}
\renewcommand {\le}{\leqslant}
\renewcommand {\geq}{\geqslant}

\RequirePackage[colorlinks,citecolor=blue,urlcolor=blue]{hyperref}

\newtheorem{Theorem}{Theorem}[section]
\newtheorem{Lemma}[Theorem]{Lemma}

\newtheorem{Prop}[Theorem]{Proposition}

\newtheorem{Rem}[Theorem]{Remark}

\newtheorem{Assume}[Theorem]{Assumption} 

\startlocaldefs

\def\Mcal{\mathcal{M}}

\newcommand{\heap}[2]{\genfrac{}{}{0pt}{}{#1}{#2}}

\newcommand{\ssup}[1] {{\scriptscriptstyle{({#1}})}}

\def\E{\mathbb E}

\def\N{\mathbb{N}}
  
\def\P{\mathbb P} 

\def\R{\mathbb{R}}

\def\e{\mathrm{e}}

\def\eps{\varepsilon}

\def\1{\mathbf{1}}
\def\3{{\ss}}

\newcommand{\ip}[1]{\left\langle #1\right\rangle}

\def\d{\mathrm{d}}

\newcommand{\gk}[1]{\left\{#1\right\}}

\newcommand\mycom[2]{\genfrac{}{}{0pt}{}{#1}{#2}}
\newcommand{\ek}[1]{\left[#1\right]}
\newcommand{\rk}[1]{\left(#1\right)}

\newcommand{\hk}[1]{^{\scriptscriptstyle{(#1)}}}
\newcommand{\abs}[1]{\left| #1 \right|}

\newcommand{\ex}{\mathrm{e}}

\newcommand{\norm}[1]{\left|\hspace{-0.395mm}\left| #1\right|\hspace{-0.395mm}\right|}

\allowdisplaybreaks[4] 

\endlocaldefs

\begin{document}

\begin{frontmatter}

\title{Commutative diagram of the Gross--Pitaevskii approximation }
\runtitle{Commutative diagrams of the Gross-Pitaevskii approximation}

\begin{aug}
\author{\fnms{Stefan}  \snm{Adams}\corref{}\ead[label=e1]{S.Adams@warwick.ac.uk}}
\and
\author{\fnms{Chiranjib} \snm{Mukherjee}\corref{}\thanksref{t1}\ead[label=e2]{chiranjib.mukherjee@uni-muenster.de}}
    \thankstext{t1}{Funded by the Deutsche Forschungsgemeinschaft (DFG) under Germany's Excellence Strategy EXC 2044--390685587, Mathematics M\"unster: Dynamics--Geometry--Structure.}
\author{\fnms{Quirin} \snm{Vogel}\corref{}\ead[label=e3]{quirin.vogel@aau.at}}
\runauthor{S.~Adams, C.~Mukherjee, Q.~Vogel}
  \affiliation{University of Warwick and University of M\"unster and University of Klagenfurt}
  
\address{S.~Adams\\Department of Mathematics,\\ University of Warwick
\\ Coventry\\ Coventry CV4 7AL\\ United Kingdom\\
          \printead{e1}\\
          }

          \address{C.~Mukherjee\\Inst.~Math.~Stochastics,\\ Department
of Mathematics\\ and Computer Science\\ University of M\"unster\\ Orl\'eans-Ring 10, D-48149\\ M\"unster, Germany\\
          \printead{e2}
          }
\par  

          \address{Q.~Vogel\\Inst.~Statistics,\\ University of Klagenfurt\\ Universit\"atsstr. 65, 9020\\ Klagenfurt, Austria\\
          \printead{e3}
          }
\end{aug}

\begin{abstract}
It is well-known that the {\it Gross--Pitaevskii} variational formula describes the ground-state energy per particle of $N$-indistinguishable trapped particles (bosons) in a dilute state in the large system size $N\to\infty$. The goal of the present article is to prove that the Gross--Pitaevskii formula also appears in the {\it iterated limit}  of zero temperature and large system size of the {\it product ground-state energy per particle} of the $N$-particle Hamiltonian operator. For fixed particle number $N$, we moreover identify the low-temperature limit of the Brownian model with mutual intersection-local-time interaction as the ground-product-state energy per particle of the corresponding $N$-body system with repulsive Dirac interaction. This proves the conjectures of Adams, K\"onig and Bru from \cite{ABK06a} and \cite{ABK06b}.

\end{abstract}



\end{frontmatter}
\begin{center}
\textit{
    Stefan Adams passed away on 8 May 2024, before the completion of the final version of this article. We deeply mourn the loss of our coauthor, colleague and mentor. His contributions to mathematical physics and his influence on the community were profound, and he is greatly missed. We dedicate this work to his memory and extend our deepest sympathies to his family.
    }
\end{center}

\section{Introduction}
The purpose of this article is to prove two conjectures posed in \cite{ABK06a} and \cite{ABK06b}. While motivated by quantum statistical mechanics, these conjectures can be formulated purely in the language of probability theory, concerning the large-deviation behavior of interacting Brownian motions. We first describe the setting and refer the reader to Section~\ref{sec:background} for the physical background.

Let $B^{\ssup 1}, \dots, B^{\ssup N}$ denote $N$ independent Brownian motions in $\R^d$ with law $\P$ and starting at the origin. We choose $\Delta$ for the generator of $B\hk{i}$, instead of $\Delta/2$, to align with the results from the mathematical physics literature. These Brownian motions interact with each other via an interaction potential $v\colon [0,\infty)\to [0,\infty]$ and are furthermore subject to a trap potential $W\colon\R^d\to [0,\infty]$, preventing the Brownian motions from escaping to infinity. Fix any {\it inverse temperature} $\beta>0$, and define the interaction of the entire ensemble
\begin{equation}\label{KN}
\begin{aligned}
\mathscr K_{N,\beta}&:= \frac 1 N \sum_{1\leq i<j \leq N} \frac 1 \beta \int_0^\beta \d s \int_0^\beta \d t N^d v\big(N|B^{\ssup i}_s- B^{\ssup j}_t|\big) \\
&\qquad\qquad\qquad\qquad+ \sum_{i=1}^N \int_0^\beta \d s \, W(B^{\ssup i}_s),
\end{aligned}
\end{equation}
and denote the logarithm of the normalizing constant as
\begin{equation}\label{F}
\mathbf F(\beta, N):=  \log \E^\P\big[\e^{-\mathscr K_{N,\beta}}\big].
\end{equation}
It was shown in \cite[Theorem 1.7]{ABK06a} that for any $N\in \N$ fixed and under suitable assumptions on $v$ and $W$ (see below), 
\begin{equation}\label{ABK1}
\lim_{\beta\to\infty}\frac 1 {\beta N} \log\E^\P\big[\e^{-\mathscr K_{N,\beta}}\big]=\lim_{\beta\to\infty}\frac 1 {\beta N} \mathbf F(\beta,N)=-\chi_N^{\ssup\otimes}\, ,
\end{equation}
where 
\begin{equation}\label{ABK2}
\chi_N^{\ssup\otimes}=  
\frac 1 N \inf_{\heap{h_1,\dots, h_N\in H^1(\R^d)}{\|h_i\|_2=1\,\forall i=1,\dots,N}} \,\,\bigg[\sum_{i=1}^N\big(\|\nabla h_i\|_2^2+ \langle W, h_i^2\rangle\big)+ \sum_{1\leq i<j \leq N} \langle h_i^2, V_N h_j^2\rangle\bigg] 
\end{equation}
and for any function $h$, we wrote $v_N(x)=N^{d-1}v(N\abs{x})$ and
\begin{equation}\label{VN}
(V_Nh)(x)= \int_{\R^d} v_N(|x-y|) h(y) \d y .
\end{equation}
  It was also shown in \cite[Theorem 1.14]{ABK06a} that the limit of the variation principle $\chi_N^\otimes$ approximates the so-called Gross--Pitaevskii formula in the large system limit  $N\to\infty$, see \cite[Theorem 1.14]{ABK06a}. Indeed, for any $d\in \{2,3\}$, 
 \begin{equation}\label{ABK3}
\lim_{N\to\infty}\bigg(\lim_{\beta\to\infty}\frac 1 {\beta N} \mathbf F(\beta,N)\bigg)= -\lim_{N\to\infty}\bigg(\chi_N^{\ssup\otimes}\bigg)=- \chi^{\ssup{\mathrm{GP}}}\, ,
\end{equation}
where $\chi^{\ssup{\mathrm{GP}}}$ is the Gross--Pitaevskii formula for $\alpha=\alpha_{\mathrm{B}}= \frac 1 {8\pi} \int_{\R^d} v(|x|) \d x<\infty$ given by
\begin{equation}
\chi^{\mathrm{\ssup{GP}}}\stackrel{\mathrm{\ssup{def}}}= \inf_{\heap{\phi\in H^1(\R^d)}{\|\phi\|_2=1}} \bigg(\|\nabla\phi\|_2^2 + \langle W,\phi^2\rangle+4\pi \alpha_{\mathrm{B}}\|\phi\|_4^4\bigg)\, ,
\end{equation}
see Section~\ref{sec:background} for more background on the Gross--Pitaevskii formula. We use the subscript $\mathrm{B}$ as $\alpha_{\mathrm{B}}$ is often called the \textit{Born}-approximation to the true scattering length, see in particular Remark~\ref{rem-LSY}.

It was conjectured in \cite{ABK06a} and \cite[p. 468]{ABK06b} that the above approximation continues to hold in the iterated limit 
$$
\lim_{\beta\to\infty}\lim_{N\to\infty} (\beta N)^{-1} \mathbf F(\beta,N),
$$ 
i.e., namely the order of the limit $\beta\to\infty$ and $N\to\infty$ is reversed in \eqref{ABK3}.

The first main goal of our present work is to provide a rigorous proof of this conjecture and show that 
$$
\lim_{\beta\to\infty}\bigg(\lim_{N\to\infty}\frac 1 {\beta N} \mathbf F(\beta,N)\bigg)= \lim_{N\to\infty}\bigg(\lim_{\beta\to\infty}\frac 1 {\beta N} \mathbf F(\beta,N)\bigg)=-\chi^{\ssup{\mathrm{GP}}} 
$$
with $\chi^{\ssup{\mathrm{GP}}}$  defined in \eqref{GP-N} for $\alpha_{\mathrm{B}}= \frac 1 {8\pi} \int_{\R^d} v(|x|) \d x<\infty$. The precise statement can be found in Theorem \ref{thm1}. 

\noindent
\textit{Dirac interaction}: since the rescaled potential in \eqref{VN} converges to the Dirac delta function, i.e., $Nv_N\to (\int v)\delta_0$, it is natural to study the case where $v$ is replaced by the delta function $\delta_0$. Consider the following interaction
\begin{equation}\label{LN}
\mathscr L_{N,\beta}=
\sum_{1\leq i < j \leq N} \frac 1 \beta\int_0^\beta\int_0^\beta \d s \d t \,\delta_0\big(B^{\ssup i}_s- B^{\ssup j}_t\big) + \sum_{i=1}^N \int_0^\beta W(B^{\ssup i}_s)\d s,
\end{equation}
see Section \ref{sec-results-2} for a precise formulation.

In this setup of singular interaction potential, it is natural to wonder if 
the (rescaled) free energy 
$$
\frac 1 {N\beta} \mathscr G_{N,\beta}:= \frac 1 {N\beta} \log \E\big[\e^{- \mathscr L_{N,\beta}}\big]
$$
converges to the analogue of \eqref{ABK2} when replacing $V$ by the delta function. 
 Indeed, for any $N\in \N$, it was conjectured in \cite[Eq. (1.35)]{ABK06a}
\begin{equation}
    \lim_{\beta\to\infty} \frac 1 {N\beta} \mathscr G_{N,\beta}= -\frac{1}{N} \inf_{\heap{h_1,\dots, h_N}{\|h_i\|_2=1\,\forall i=1,\dots,N}} \,\,\bigg[\sum_{i=1}^N\big(\|\nabla h_i\|_2^2+ \langle W, h_i^2\rangle\big)+ \sum_{1\leq i<j \leq N} \langle h_i^2,  h_j^2\rangle\bigg].
\end{equation}
Our second main result, stated in Section \ref{sec-results-2}, contains in particular a proof of this conjecture, see Theorem \ref{thm2}.

\section{Main results}\label{sec-results}
In this section we rigorously state the main results. See Section~\ref{sec:background} for a discussion of their physical relevance.

\subsection{\bf {Gross--Pitaevskii formula in the commutative limit}}\label{sec-results-1}

From now on we fix a spatial dimension $d\in \{2,3\}$. Recall the definition of the free energy $\mathbf  F(\beta,N)=\log \E\big[\e^{-\mathscr K_{N,\beta}}\big]$ from \eqref{KN} for the rescaled 
pair interaction potential $v_N(|x|)= N^{d-1} v(N |x|)$.

\begin{Assume}\label{assumption}
We impose the following assumptions on the trap potential $W$ and the interaction potential $v$, respectively. 

\begin{itemize}
\item Assume that $W:\R^d \to [0,\infty]$ is continuous on the connected set $\{W<\infty\}$ and $\lim_{\abs{x}\to\infty}W(x)=+\infty$.

\item Let $v: [0,\infty)\to [0,+\infty]$ be measurable, bounded from below and continuous on $\{v<\infty\}$ with furthermore $v(|\cdot|)\in L^1(\R^d)\cap L^2(\R^d)$. Moreover, $v(r)<\infty$ for all $r>0$ as well as $\int_{\R^d} v(|x|)\d x>0$. 
\end{itemize}

\end{Assume}

We are now ready to state our first main result.

\begin{Theorem}\label{thm1}
Fix $d\in \{2,3\}$ and let Assumption \ref{assumption} be satisfied. Then for $\alpha_{\mathrm{B}}=\frac 1 {8\pi} \int_{\R^d} v(|x|) \d x<\infty$, 
\begin{equation}
\begin{aligned}
\lim_{\beta\to\infty}\bigg(\lim_{N\to\infty}\frac 1 {\beta N} \mathbf F(\beta,N)\bigg)&= \lim_{N\to\infty}\bigg(\lim_{\beta\to\infty}\frac 1 {\beta N} \mathbf F(\beta,N)\bigg)\\
&=-\inf_{\heap{\phi\in H^1(\R^d)}{\|\phi\|_2=1}} \bigg(\|\nabla\phi\|_2^2 + \langle W,\phi^2\rangle+4\pi \alpha_{\mathrm{B}}\|\phi\|_4^4\bigg)\\
&=-\chi^{\ssup{\mathrm{GP}}}. 
\end{aligned}
\end{equation}
\end{Theorem}

\subsection{\bf Ground product states for the Dirac interaction in the low temperature regime $\beta\to\infty$}\label{sec-results-2}

\medskip

\noindent We will now state our main results concerning the interaction $\mathscr L_{N,\beta}$ defined in \eqref{LN}. As already remarked, this is only a formal expression and a precise meaning is given by the {\it Brownian intersection local time} which is defined as follows. Fix $1\leq i<j\leq N$ and consider the process $X^{\ssup{i,j}}_{st}=B^{\ssup{i}}_s-B^{\ssup{j}}_t$, the so-called {\em confluent Brownian motion\/} of $B^{\ssup{i}}$ and $B^{\ssup{j}}$. It is known \cite[Th.~1]{GHR84} that this two-parameter process possesses a normalized {\it local time}, i.e., for $d\in\{2,3\}$ there is a random process $(L^{\ssup{i,j}}_\beta(x))_{x\in\R^d}$ such that, 
$x\mapsto L_\beta^{\ssup{i,j}}(x)$ may be chosen to be continuous and moreover, 
for any bounded and measurable function $f\colon\R^d\to \R$, 
\begin{equation}\label{intersection}
\begin{aligned}
\int_{\R^d}f(x)L^{\ssup{i,j}}_\beta(x)\,\d x&=\frac 1{\beta^2}\int_0^\beta\d s\int_0^\beta\d t\, f\bigl(B_s^{\ssup{i}}-B_t^{\ssup{j}}\bigr)\\
&=\int_{\R^d}\int_{\R^d}\mu_\beta^{\ssup{i}}(\d x)\mu_\beta^{\ssup{j}}(\d y) f(x-y). 
\end{aligned}
\end{equation}
where 
\begin{equation}\label{mu}
\mu_\beta^{\ssup i}=\frac 1 \beta\int_0^\beta \delta_{B^{\ssup i}_s} \d s\in \mathcal M_1(\R^d)
\end{equation}
 is the normalized empirical measure (or the {\it occupation measure}) of the $i^{\mathrm{th}}$ Brownian motion, which is a random element 
 of the space $\mathcal M_1(\R^d)$ of probability measures on $\R^d$. Both the empirical measures $\mu_\beta^{\ssup i}$ as well as the 
 intersection local time $L^{\ssup{i,j}}_\beta$ will play useful roles later on in the article. 
 
 Note also that by the aforementioned continuity property, $L^{\ssup{i,j}}_\beta(0)$ is well-defined, and is just the (normalized) amount of interaction of the trajectories $B^{\ssup i}$ and $B^{\ssup j}$ until time $\beta$. Hence, we rewrite \eqref{LN} precisely as 
\begin{equation}\label{LN2}
\begin{aligned}
\mathscr L_{N,\beta}= \beta   \sum_{1\leq i < j \leq N} L^{\ssup{i,j}}_\beta(0)+
\sum_{i=1}^N \int_0^\beta W(B^{\ssup i}_s)\d s. 
\end{aligned}\end{equation}
Recall that $\mathscr G_{N,\beta}=\log \E\big[\e^{-\mathscr L_{N,\beta}}\big]$. Given \eqref{ABK1} and the results stated in Section \ref{sec-results-1}, a natural question is to determine the   rescaled free energy $\lim_{\beta\to\infty}\frac 1 {N\beta} \mathscr G_{N,\beta}$ in the zero-temperature limit for any fixed $N\in\N$. 

\begin{Theorem}\label{thm2}
Suppose that $W$ satisfies the conditions imposed in Assumption \ref{assumption}. For $d\in\{2,3\}$ define
\begin{equation}
    \chi_N^{\ssup{\delta_0}}= \frac 1 N \inf_{\heap{h_1,\dots, h_N\in H^1(\R^d)}{\|h_i\|_2=1\,\forall i=1,\dots,N}} \,\,\bigg[\sum_{i=1}^N\big(\|\nabla h_i\|_2^2+ \langle W, h_i^2\rangle\big)+ \sum_{1\leq i<j \leq N} \langle h_i^2,  h_j^2\rangle\bigg].
\end{equation}
Then for any fixed $N\in \N$, 
\begin{equation}\label{eq-thm2}
\lim_{\beta\to\infty} \frac{1}{\beta N} \mathscr G_{N,\beta}= -\chi_N^{\ssup{\delta_0}}\, .
\end{equation}
\end{Theorem}

\section{Background and motivation}\label{sec:background}
In this section, we describe the motivation to study systems of interacting Brownian motions such as \eqref{KN}.
\subsection{\bf The Gross--Pitaevskii formula} 

Consider an $N$-particle quantum system which is described by an $N$-particle Hamiltonian operator 
$$
\begin{aligned}
 H_N= - \sum_{i=1}^N \Delta_i+ \sum_{1\leq i < j \leq N} v(|x_i-x_j|) 
+ \sum_{i=1}^N W(x_i), \quad\!\! (x_1,\dots,x_N)\in (\R^d)^N.
\end{aligned}
$$
In the present setup, the kinetic energy term $\Delta_i$ acts on the $i^{\mathrm{th}}$ particle,  $v:(0,\infty)\to (0,\infty)$ is a potential 
which captures  {\it mutual (or pair) interaction} between two particles, decays fast at infinity and explodes close to the origin, while $W:\R^d \to [0,\infty]$ is a {\it trap potential} which tends to keep the quantum particles confined in a bounded region (e.g. $W(x)=|x|^2$). An important task in quantum statistical mechanics is a complete description of such a particle system at zero (or very low) temperature in the many-particle limit $N\to\infty$ (or $\Lambda\to\R^d$). The incentive for such a task is motivated by the quest for a rigorous understanding of the emergence of {\it Bose--Einstein condensation} (BEC), which is a physical phenomenon 
introduced in the work of S.N. Bose in 1924 and followed up by the predictions of Einstein in 1925. For a system of particles obeying Bose statistics with conserved particle number, Bose–Einstein condensation means that, below a critical temperature, a macroscopic fraction of the particles occupies one one-particle state.

Rigorous mathematical pursuit pertinent to BEC started in 1940s with the 
works of Bogoliubov and Landau as well as Penrose and Feynman, see \cite{AB04a,AB04b} for a review. 
Another sound mathematical approach\footnote{there is an argument to be made that the setup in the laboratory which induces Bose--Einstein condensation is more akin to a dilute limit, see \cite{R20}} to analyze the above quantum system involves studying the
 system in the ``dilute regime"  on a particular scale: physically, diluting the system corresponds to 
 keeping the particles confined in a bounded region (e.g. by the presence of a trap $W$ as above) 
 and maintaining the range of the inter-particle distance small compared to the mean particle distance, e.g. 
 for a suitable constant $a_N=a_N(d,v)\to 0$ as $N\to\infty$, the pair interaction potential $v$  is replaced by its rescaled version $v_N(\cdot)=a_N^{-2} v(\cdot /a_N)$ leading to the rescaled Hamiltonian 
\begin{equation}\label{HN}
\begin{aligned}
 \mathscr H_N= - \sum_{i=1}^N \Delta_i+ \sum_{1\leq i < j \leq N} v_N(|x_i-x_j|) 
+ \sum_{i=1}^N W(x_i),
\end{aligned}
\end{equation}

In this regime, the ground state and its energy for $\mathscr H_N$ were studied 
in the many-particle limit $N\to\infty$ in a series of works of Lieb, Seiringer, Solovej  and Yngvason (\cite{LSY00a,LY01,LSY01}).
Indeed, let  
\begin{equation}\label{chiN}
\begin{aligned}
\chi_N&:=\frac 1 N \inf_{\heap{h\in H^1(\R^{dN})}{ \|h\|_{2}=1}} \, \big\langle h, \mathscr H_N h\big\rangle \\
&= \frac 1 N\inf_{\heap{h\in H^1(\R^{dN})}{ \|h\|_{2}=1}} \bigg[ \sum_{i=1}^N\big(\|\nabla_ih\|_2^2+ \langle h^2, W(x_i)\rangle\big)\\
&\qquad\qquad\qquad\qquad+ \sum_{1\leq i < j \leq N}\big\langle h^2, v_N(|x_i-x_j|) \big\rangle\bigg]
\end{aligned}
\end{equation} 
be the {\it ground-state energy per particle}.  Then 
it was shown in the aforementioned works 
that in $d=2,3$ and under appropriate choice of $a_N=a_N(d,v)$\footnote{In the most physically relevant case $d=3$, $a_N$ is chosen to be $a_N=1/N$, see Remark \ref{rem-LSY}.} and suitable assumptions on $v$ and $W$ (see Remark \ref{rem-LSY}), the well-known {\it Gross--Pitaevskii formula} adequately describes the ground-state energy per particle $\chi_N$  in the limit $N\to\infty$:
\begin{equation}\label{GP-N}
\lim_{N\to\infty} \chi_N = \chi^{\mathrm{\ssup{GP}}}\stackrel{\mathrm{\ssup{def}}}= \inf_{\heap{\phi\in H^1(\R^d)}{\|\phi\|_2=1}} \bigg(\|\nabla\phi\|_2^2 + \langle W,\phi^2\rangle+4\pi \alpha\|\phi\|_4^4\bigg)
\end{equation}
The variational formula on the right hand side is known as the {\it Gross--Pitaevskii} formula, which  was derived by Gross and Pitaevskii independently in 1961, based on the aforementioned work of Bogoliubov and Landau. In this variational formula, the pair-interaction function $v$ manifests only in the parameter $\alpha=\alpha(v)$, known as the {\it scattering length}, see Remark \ref{rem-LSY}. \footnote{As per physical prediction, $\alpha$ is the only parameter defined by $v$ that persists in the limit $N\to\infty$ as a pre-factor in the quartic term of $\chi^{\mathrm{\ssup{GP}}}$.} 
Furthermore, minimizers  $h_N^\star$ of the variational formula \eqref{chiN}  exist and are called the {\it ground states} of the Hamiltonian $\mathscr H_N$. Moreover, there is a unique minimizer $\phi^{\mathrm{\ssup{GP}}}=\phi^{\mathrm{\ssup{GP}}}(\alpha)$ of the Gross--Pitaevskii formula \eqref{GP-N} which is smooth, strictly positive and bounded (\cite[Theorem 2.1]{LSY00a}). More precisely, it was also shown in \cite{LSY00a,LY01,LSY01} that complete condensation is expressed by convergence of the normalized one-particle reduced density matrix to the rank-one projection onto the Gross--Pitaevskii minimizer.

\subsection{\bf Gross--Pitaevskii formula and the ground-product state energy of $\mathscr H_N$}

The link with our probabilistic analysis now becomes apparent: restricting \eqref{chiN} to \textit{product states}, i.e., functions $h$ with $h\in H^{1}(\R^{dN})$ of the form $h=h_1\otimes\ldots \otimes h_N$ with $h_i\in H^{1}(\R^{d})$ gives the limit of the free energy of the interaction function defined in \eqref{KN}, rewriting now \eqref{ABK2} as
\begin{equation}\label{eq:productAnsatz}
    \frac 1 N \inf_{\heap{h_1,\dots, h_N}{\|h_i\|_2=1\,\forall i=1,\dots,N}}  \,\, \big\langle h_1\otimes\dots\otimes h_N, \mathscr H_N(h_1\otimes\dots\otimes h_N)\big\rangle \, .
\end{equation}
Note however that Bosonic wavefunctions are symmetric by the Spin–statistics theorem. If the $h_i$'s are distinct, the tensor product is hence not boson-symmetric. Thus this is a product-state variational model, not literally the restriction to symmetric bosonic pure product states.

\subsection{\bf Delta-interaction ground-product-state energy} 

A further interesting choice of the interaction potential in the Hamiltonian $\mathscr H_N$ in \eqref{HN}, can be made when we no longer choose the interaction as a function but as a {\it measure}, while still preserving the singularity at zero. Indeed, recall that the rescaled interaction is given by $v_N(x)=N^{d-1}v(N|x|)$. The induced measure
\begin{equation}
    Nv_N(x)\d x=N^{d}v(N|x|)\d x\, ,
\end{equation}
converges weakly as $N\to\infty$ to
\begin{equation}
    \rk{\int_{\R^d}v(|x|)\d x}\delta_0\, .
\end{equation}
This suggests considering, already at fixed particle number, a product-state model in which the pair-interaction kernel is replaced by a contact measure $\lambda\delta_0$ for some $\lambda>0$. In the variational principle \eqref{eq:productAnsatz}, using the approximation $v_N\approx \lambda\delta_0$ leads to an interaction term
\begin{multline}
    \sum_{1\le i<j\le N}\int_{\R^d}\int_{\R^d}Nv_N(x-y)\abs{h_i(x)}^2\abs{h_j(y)}^2\d x \d y\\
    \approx  \lambda\sum_{1\le i<j\le N}\int_{\R^d}\abs{h_i(x)}^2\abs{h_j(x)}^2\d x \, .
\end{multline}
For $\lambda>0$ and
$(h_1,\ldots,h_N)\in H^1(\R^d)^N$, one can hence define the
contact-interaction ground-product-state functional by
\begin{equation}\label{eq:contact-Hartree-functional}
\begin{aligned}
\mathcal E^{\mathrm{ct}}_{N,\lambda}
(h_1,\ldots,h_N)
={}&
\sum_{i=1}^N
\left(
\|\nabla h_i\|_2^2
+
\langle W,|h_i|^2\rangle
\right)
\\
&+
\lambda
\sum_{1\leq i<j\leq N}
\int_{\R^d}
|h_i(x)|^2|h_j(x)|^2\,\d x .
\end{aligned}
\end{equation}
Its ground-product-state energy per particle is
\begin{equation}\label{eq:contact-Hartree-energy}
\chi_N^{(\delta_0)}(\lambda)
=
\frac1N
\inf_{\substack{h_1,\ldots,h_N\in H^1(\R^d)\\
\|h_i\|_2=1,\ i=1,\ldots,N}}
\mathcal E^{\mathrm{ct}}_{N,\lambda}
(h_1,\ldots,h_N).
\end{equation}
Note that it is well defined by the Sobolev embedding theorem. Furthermore, note that
\begin{equation}
    \sum_{i<j}\int \abs{h_i}^2\abs{h_j}^2=\frac{1}{2}\int\rk{\rk{\sum_{i=1}^N \abs{h_i}^2}^2-\sum_{i=1}^N \abs{h_i}^4}\, ,
\end{equation}
encourages the $h_i$'s to place their mass in well-separated regions; this is the repulsive nature of the interaction.

Choosing now $h_1=\ldots=h_N=\phi$ in \eqref{eq:contact-Hartree-energy} yields
\begin{equation}
    \frac{1}{N}\mathcal E^{\mathrm{ct}}_{N,\lambda}(\phi,\ldots,\phi)=\norm{\nabla \phi}_2^2+\ip{W,\phi^2}+\frac{\lambda(N-1)}{2}\norm{\phi}_4^4\, .
\end{equation}
Choosing $\lambda=\lambda_N$ of mean-field scaling with $\lambda_N=\frac{8\pi \alpha_{\mathrm{B}}}{N-1}$ yields the Gross--Pitaevskii functional from the introduction.

\begin{Rem}
The notation $\delta_0(x_i-x_j)$ is used here at the variational problem level.  We do not define the corresponding
many-body operator by naively adding the measure $\delta_0(x_i-x_j)$ to the Laplacian.  In dimensions two and three,
such point interactions require a separate operator-theoretic construction. The product-state functional \eqref{eq:contact-Hartree-functional}, by contrast, is directly
well defined.
\end{Rem}

\subsection{\bf Physical relevance of Theorem~\ref{thm1}} 
Let us underline the relevance of Theorem \ref{thm1} in the present context. As already remarked earlier, the zero-temperature limit of the rescaled free energy $-\frac 1 {\beta N}\mathbf  F(\beta,N)$ is described by the product ground-state energy per particle $\chi^{\ssup\otimes}_N$. The replacement of the ground-state energy per particle $\chi_N$ (defined in \eqref{chiN}) by its product state counterpart $\chi_N^{\ssup\otimes}$ is known as the {\it Hartree approach} and the variational formula $\chi_N^{\otimes}$ is called the {\it Hartree formula}, see the physics monograph \cite[Ch. 12]{DvN05}.
The limiting assertion \eqref{ABK3} then shows  that the Hartree formula too approximates the Gross--Pitaevskii limit as $N\to\infty$. In this context, Theorem \ref{thm1} underlines that this approximation of the Gross--Pitaevskii formula is {\it stable under the iteration of the limits $\beta\to\infty$ and $N\to\infty$}, i.e. the diagram 

\begin{tikzpicture}
  \matrix (m) [matrix of math nodes,row sep=5em,column sep=8em,minimum width=4em]
  {
     -\frac 1{N\beta} \mathbf  F(\beta,N) & \chi^{\ssup\otimes}_N \\
     \chi^{\ssup\otimes}(\beta) & \chi^{\ssup{\mathrm{GP}}} \\};
  \path[-stealth]
   (m-1-1)  edge [left] node  {$\,\,\, N\to\infty$} (m-2-1)
            edge [right] node [above] {$\beta\to\infty$} (m-1-2)
    (m-2-1.east|-m-2-2) edge node [below] {$\beta\to\infty$} (m-2-2)
    (m-1-2) edge node [right] {$N\to\infty$} (m-2-2);
            \end{tikzpicture}

\noindent actually commutes, see \eqref{eq:chibeta} for the definition of the large scale limit $\chi^{\otimes}(\beta)$ at positive temperature.  Such an assertion also reconfirms the physical intuition that the Hartree model is a suitable ansatz for a rigorous understanding of the so-called {\it trace formula} of the {\it canonical ensemble model}, which is akin to the Hartree model discussed so far. Indeed, the bottom of the spectrum of $\mathscr H_N$ can be recovered from the low-temperature asymptotics of the trace $\e^{-\beta\mathscr H_N}$, i.e. $\chi_N=-\lim_{\beta\to\infty} \frac1 {N\beta} \log\mathrm{Tr}[\e^{-\beta \mathscr H_N}]$, see Ginibre \cite{G71}. Now the Feynman-Kac formula (for traces) implies that $\mathrm{Tr}[\e^{-\beta \mathscr H_N}]= \int \d x \E_x^{\ssup\beta}\big[\e^{-\beta\langle \mathscr W+ \mathscr V, \mu_\beta\rangle}\big]$ where $\E^{\ssup\beta}_x$ denotes expectation w.r.t. a Brownian bridge $\bar B_s=(B^{\ssup1}_s,\dots,B^{\ssup N}_s)$ in $\R^{dN}$ pinned at $x$ in the time interval $[0,\beta]$, $\mu_\beta=\frac 1 \beta\int_0^\beta \delta_{\bar B_s}\d s$ and $\mathscr W(x)=\sum_{i=1}^N W(x_i)$ and $\mathscr V(x)=\sum_{i<j} v(|x_i-x_j|)$. Thus (see also \cite[Theorem 1.5]{ABK06a}) 
$$
\chi_N= \lim_{\beta\to\infty}  \frac1 {N\beta} \log \mathsf{E}\bigg[\e^{-\sum_{1\leq i < j \leq N} \int_0^\beta \d s \,  v(|B^{\ssup i}_s- B^{\ssup j}_s|) - \sum_{i=1}^N\int_0^\beta W(B^{\ssup i}_s)}\bigg]\, ,
$$
where $\mathsf{E}$ now also averages over all starting positions of the bridges.

The logarithm of the expectation on the right hand side above is called the free energy of the {\it canonical ensemble model}, which is closely related to the Hartree model defined in \eqref{KN}, except for that the pair-potential of the Hartree models captures interactions of the {\it trajectories}, while the canonical ensemble model above is defined via interactions  of the {\it particles}. However, like the Hartree model, the {\it zero temperature} limit of the rescaled free energy of the canonical ensemble model above also converges to the Gross--Pitaevskii formula as $N\to\infty$, recall \eqref{GP-N}. On the other hand, for a {\it fixed positive temperature}, investigation of the large system limit of this free energy is an important open problem. However, the stability of the approximation in Theorem \ref{thm1} is an instructive rigorous step towards a full understanding of the desired limiting scheme 
$$\lim_{\beta\to\infty}\bigg(\lim_{N\to\infty} \frac1 {N\beta} \log\mathrm{Tr}[\e^{-\beta \mathscr H_N}]\bigg)
$$ 
of the canonical ensemble model. We refer to \cite{BKM24,Koe26} for recent progress in that direction.

\subsection{\bf Relevant remarks and existing literature}\label{sec-remarks}

\begin{Rem}[On the Assumptions on $v$ and $W$]
The continuity assumptions on $W$ and $v$ are needed for the proof of the limiting upper bound for \eqref{ABK1} as $\beta\to\infty$ which relies on large deviation arguments. 

Indeed, so far the results have been stated and proved in the presence of a trap term $W$. In this setting, the probabilistic approaches used in \cite{ABK06a} are based on
applying Donsker-Varadhan theory of large deviations for the distribution of Brownian occupation measures which will no longer hold true if we assume $W\equiv 0$. 
In this case, since all the models which have been discussed are shift-invariant functionals of the occupation measures $(\mu^{\ssup i}_\beta,\mu^{\ssup j}_\beta)$ for $i\ne j$, 
similar statements can be derived using the theory developed recently in \cite{MV16} and its extension \cite{Muk17} pertaining to compactification of orbits spaces of probability measures and large deviation theory therein. The theory \cite{MV16} has since found a wide range of applications to large-deviation and variational problems arising in statistical mechanics. In particular, it has been applied to the resolution of the Polaron problem \cite{MV16,KM15,BKM15,MV18a,MV18b,BMSV23}, directed polymers and Gaussian multiplicative chaos \cite{BC16,BM18,BS20,BLM22}, as well as large deviations for the Wiener sausage and the Swiss cheese problem \cite{EP26,ETS25,EP23}. More recently, it has also found applications to compactifications of Schur multipliers of Roe algebras and group-invariant percolation \cite{MR24}.

\end{Rem}

\begin{Rem}[On dimensions $d\in \{2,3\}$]\label{rem-d}
For any fixed system size $N\in \N$, the zero temperature limit \eqref{ABK1} was shown to hold in any dimension $d\geq 1$. Now recall that the limit $N\to\infty$ in \eqref{ABK3} holds under the rescaling $v_N= N^{d-1} v(N\cdot)$ leading to \eqref{KN}. Since $N^d v(N\cdot)$ is an approximation of the Dirac measure at $0$, 
the double integral $\int_0^\beta\int_0^\beta N^d v\big( N|B^{\ssup i}_s-B^{\ssup j}_t|\big)$ is an approximation of the {\it intersection local time} 
$\int_0^\beta\int_0^\beta \delta_0(B^{\ssup i}_s-B^{\ssup j}_t)$ of $B^{\ssup i}$ and $B^{\ssup j}$, and thus the convolution integral w.r.t. $V_N$ (recall \eqref{VN} and \eqref{ABK2}) also converges formally to the quartic term in the Gross--Pitaevskii formula. The intersection local time is a measure which is supported on the set of (mutual) intersections of $B^{\ssup i}$ and $B^{\ssup j}$ and can be defined rigorously in $d\in \{2,3\}$. This measure also manifests in the limit $N\to\infty$ of $\frac 1 N \mathbf F(N,\beta)$ for a fixed positive temperature $\beta$ in the variational formula $\chi^{\ssup\otimes}(\beta)$, see \eqref{eq:chibeta} and the discussion that follows. 
\end{Rem}

\begin{Rem}[The scattering length]\label{rem-LSY}
Recall that the approximation \eqref{GP-N} was shown in \cite{LSY00a,LY01,LSY01} in $d\in \{2,3\}$ assuming that $v\geq 0$ with $v(0)>0$ and $\int_{1}^\infty v(r) r^{d-1} \d r<\infty$ and choosing $v_N(\cdot)= \beta_N^{-2} v(\cdot \beta_N^{-1})$ where $\beta_N=1/N$ in $d=3$. Furthermore, for $d=3$, $\alpha(v)=\lim_{r\to\infty}\big[r- \frac {u(r)}{u'(r)}\big]\in (0,\infty)$ with $u^{\prime\prime}=\frac 12 uv$ on $(0,\infty)$ and $u(0)=0$. In $d=3$, it is known that 
$\alpha(v) < \frac 1 {8\pi} \int_{\R^d} v(|x|) \d x=\alpha_{\mathrm{B}}$ (recall Theorem \ref{thm1} and note that the latter integral is also referred to as the {\it first Born approximation of the scattering length} of $v$, see \cite{LSSY05}). In $d=2$, 
$\beta_N=(\alpha(v)^{-1} \|\phi^{\ssup{\mathrm{GP}}}\|_4^{-2}) N^{1/2}\e^{-N/2\alpha(v)}$ where $\phi^{\ssup{\mathrm{GP}}}$ is the unique minimizer of the Gross--Pitaevskii formula and $\alpha=\alpha(v)$ is the {\it scattering length} which is defined as follows: if $v$ has compact support in $[0,R_\star]$ then 
the scattering length is defined as $\log\alpha(v)= \frac{\log r- u(r)\log R}{1-u(r)}$ for $r \in (R_\star,R)$ with $u^{\prime\prime}(r)+\frac{1}{r}u^\prime(r)=\frac 12 u(r)v(r)$ on $[0,R]$ with $u(R)=1$ and $u'(0)=0$. In case $v$ does not have compact support, the scattering length is defined as a limit obtained from approximating $v$ by a compactly supported function. 
\end{Rem}

\begin{Rem}[Reduced density matrix and BEC]
We recall the variational formula $\chi_N$ defined in \eqref{chiN}. Then $\chi_N$ possesses a unique minimizer $h^\star_N$ which defines the so-called {\it reduced density matrix} as
\begin{equation}
    \gamma_N(x,y):=N\int_{(\R^d)^{N-1}} h_N^\star(x,x_2,\dots,x_N) \overline{h_N^\star(y,x_2,\dots,x_N)}\d x_2\dots \d x_N\, .
\end{equation}
In the Gross–Pitaevskii limit, the one-particle reduced density matrix (scaled by $\frac{1}{N}$) of the many-body ground state converges in trace norm to the rank-one projection onto the Gross–Pitaevskii minimizer, expressing complete Bose–Einstein condensation, see \cite{LSY00a,LY01,LSY01}.
\end{Rem}

\begin{Rem}[Probabilistic approaches to BEC]
    Probabilistic approaches to Bose--Einstein condensation have made considerable progress in recent years. A milestone was a proof of macroscopic cycles in the random interchange process, see \cite{ES20}, which is an approximation to the true quantum mechanical model. For a time-discretized approximation of the interacting Bose gas, condensation was proven in \cite{QT23}. The reduced density matrix was studied using the Feynman--Kac representation in \cite{KVZ25,BKV25}, limited to the case $v\equiv 0$. For a link between the emergence of macroscopic loops and BEC, see \cite{AFY21} are reference therein.
\end{Rem}

\section{Proof of Theorem \ref{thm1}}\label{sec-pf-thm1}

The proof of Theorem \ref{thm1} will be carried out in several main steps. We begin with a technical lemma, proving that our assumptions on $v$ comply with \cite{ABK06a}
\begin{Lemma}\label{lem:uniform-kato}
Let $d\in\{2,3\}$ and suppose that
$v(|\cdot|)\in L^2(\R^d)$. Set
\[
\kappa_d(x)=
\begin{cases}
1+\log(1\vee |x|^{-1}),&d=2,\\
|x|^{-1},&d=3.
\end{cases}
\]
Then
\[
\lim_{r\downarrow0}
\sup_{w\in\R^d}
\int_{B_r(0)}
\kappa_d(y)v(|y-w|)\,\d y=0.
\]
In particular, the convolution potentials
\[
V\mu(x)=\int_{\R^d}v(|x-y|)\,\mu(\d y),
\qquad \mu\in\mathcal M_1(\R^d),
\]
belong uniformly to the Brownian Kato class.
\end{Lemma}
\begin{proof}
    This directly follows from the Cauchy--Schwarz inequality and the local square-integrability of $\kappa_d$.
\end{proof}
The above lemma shows that $v$ satisfies the assumptions of \cite[Theorem 1.7]{ABK06a}\footnote{Note that \cite[Theorem 1.7]{ABK06a} mentions the Green's function $G$ of the free Brownian motion for $d=2$. However, this is a mistake, as $G=+\infty$ in the planar case. The proof of \cite[Theorem 1.7]{ABK06a} is correct notwithstanding, as they only use that $V\mu$ is Kato class.}. Let us now assume that $\{W<\infty\}=\R^d$, to simplify notation.

Now, we set-up some relevant notation and collect some preliminary facts which will be useful in the sequel. 
Let $C_b(\R^d)$ denote the space of continuous and bounded functions, while $\mathcal M_1(\R^d)$ denotes the space of probability measures on $\R^d$.  For any $\beta>0$, let us first introduce the energy functional $\mathscr J_\beta: \mathcal M_1(\R^d)\to [0,\infty]$ defined by
\begin{equation}\label{J}
\mathscr J_\beta(\mu)= \sup_{f\in C_b(\R^d)} \bigg[ \langle \mu, f\rangle - \frac 1 \beta \log \E^\P\ek{\e^{\int_0^\beta f(B_s)\d s }}\bigg].
\end{equation} 
For any $\alpha, \beta>0$, now let us define the variational formula 
\begin{equation}\label{eq:chibeta}
\chi^{\ssup\otimes}(\beta)= \inf_{\heap{\phi\in L^2(\R^d)\cap L^4(\R^d)}{\|\phi\|_2=1}} \bigg[ \mathscr J_\beta(\phi^2) + \langle W,\phi^2\rangle+ 4\pi\alpha\|\phi\|_4^4 \bigg].
\end{equation} 
In the formula above, we wrote $\mathscr J_\beta(\phi^2)= \mathscr J_\beta(\mu)$ for $\mu(\d x)=\phi^2(x)\d x$ for $\|\phi\|_2=1$. 

The three terms in the variational formula $\chi^{\ssup\otimes}(\beta)$ can be interpreted as follows. The term $\langle W,\phi^2\rangle$ is the energy cost from the paths for staying constrained in a bounded region enforced by the trap potential $W$. The quartic term $\int \phi^4$ is a manifestation of the limiting effective interaction captured by the aforementioned intersection local time $\int_0^\beta\int_0^\beta \delta_0(B^{\ssup i}_s- B^{\ssup j}_t)$ 
coming from $N^d v(N\cdot)\Rightarrow \delta_0$, recall Remark \ref{rem-d}. The energy functional $\mathscr J_\beta$ is akin to a (relative) entropy term, which can be read off naturally as the optimal cost paid by the averages of the empirical distributions of the Brownian paths which satisfy a Cram\'er type large deviation principle with a rate function which is 
the Legendre–Fenchel transformation of the logarithmic moment generating functional of the empirical measure of a single Brownian path, see the proof of Proposition \ref{propABK06b} below. 

The rest of Section \ref{sec-pf-thm1} is devoted to the proof of Theorem \ref{thm1}. We will recall the following useful fact, which follows standard arguments in large deviation theory. 
\begin{Lemma}\label{lemma-phi}
There exists $\mu\in \mathcal M_1(\R^d)$ such that $\mathscr J_\beta(\mu)<\infty$. 
Moreover, $\mathscr J_\beta(\mu)=\infty$ unless $\mu$ is absolutely continuous w.r.t. the Lebesgue measure with density $\phi^2(x)=\frac{\d\mu}{\d x}$ and $\|\phi\|_2=1$. Finally, there exists a unique non-negative 
$\phi_\star\in L^4(\R^d) \cap L^2(\R^d)$ such that $\phi_\star^2$ is a minimizer of the variational formula \eqref{eq:chibeta}. 
\end{Lemma}
\begin{proof}
The functional $\mathscr J_\beta$ was studied in the \cite{ABK06b}. The existence of $\mu\in \mathcal M_1(\R^d)$ such that $\mathscr J_\beta(\mu)<\infty$ was shown in \cite[Lemma 2.1]{ABK06b}. In \cite[Lemma 2.2]{ABK06b}, it was shown that $\mathscr J_\beta(\mu)=\infty$ unless $\mu$ is absolutely continuous w.r.t. the Lebesgue measure with density $\phi^2(x)=\frac{\d\mu}{\d x}$. 
Furthermore, the uniqueness of the non-negative minimizer 
$\phi_\star\in L^4(\R^d) \cap L^2(\R^d)$ follows from the strict convexity of the map $\phi^2\mapsto\norm{\phi}_4^4$.
\end{proof}

\subsection{The many particle limit of the rescaled free energy at fixed temperature}

An important ingredient for the proof of Theorem \ref{thm1} is 
\begin{Prop}\label{propABK06b}
Fix $\beta>0$ and suppose that Assumption \ref{assumption} holds. Then for $\alpha_{\mathrm{B}}= \frac 1 {8\pi} \int_{\R^d} v(|x|) \d x$, 
$$
\lim_{N\to\infty}\frac 1 {\beta N} \mathbf F(\beta,N)= - \chi^\otimes(\beta).
$$
\end{Prop}
which is given in \cite[Proposition 3.1]{ABK06b}.

\subsection{The variational formula $\chi^{\ssup{\otimes}}(\beta)$ in the zero-temperature limit $\beta\to\infty$}

Recall the variational formula $\chi^{\ssup\otimes}(\beta)$ from \eqref{eq:chibeta}. The goal of this section is to prove 

\begin{Theorem}\label{thm3}
The following holds:
\begin{equation}
\lim_{\beta\to\infty} \chi^{\ssup\otimes}(\beta)= \inf_{\heap{\phi\in H^1(\R^d)}{\|\phi\|_2=1}} \bigg(\|\nabla\phi\|_2^2 + \langle W,\phi^2\rangle+4\pi \alpha_{\mathrm{B}}\|\phi\|_4^4\bigg)
=\chi^{\ssup{\mathrm{GP}}}
\end{equation}
\end{Theorem} 
Before proving the above theorem, we need to prove some preliminary results.

Set $\Lambda_\beta(f)=\frac{1}{\beta}\log\E\ek{e^{\int_0^\beta f(B_s)\d s}}$ and set 
\begin{equation}
    \Mcal_{W,1}=\gk{\phi\in L^2(\R^d)\colon \norm{\phi}_2=1\quad \textnormal{and}\quad \ip{\phi^2,W}<\infty}\, .
\end{equation}
By \cite[Lemma 2.3]{ABK06b}, for $\phi\in \Mcal_{W,1}$
\begin{equation}
    \mathscr J_\beta(\phi^2)=\sup_{f\in C_b(\R^d)}\gk{\ip{f-W,\phi^2}-\Lambda_\beta(f-W)}\, .
\end{equation}
Abbreviate
\begin{equation}
    M_\beta(f)=\Lambda_\beta(f-W)=\frac{1}{\beta}\log\E\ek{\ex^{\int_0^\beta (f-W)(B_s)\d s}}\, .
\end{equation}
Next, we compute the limit of $M_\beta(f)$.
\begin{Lemma}\label{lem:FKconvergence}
    For all $f\in C_b(\R^d)$
    \begin{equation}\label{eq:842026}
        \lim_{\beta\to\infty}M_\beta(f)=M(f):=\sup_{\mycom{\phi\in H^1(\R^d)}{\norm{\phi}_2=1}}\gk{\ip{f-W,\phi^2}-\norm{\nabla \phi}_2^2}\, .
    \end{equation}
\end{Lemma}
\begin{proof}
The claim follows directly from \cite[Proposition 5.1]{ABK06a} with $N=1$ and $v=0$.
\end{proof}
According to \eqref{eq:chibeta}, it is relevant to consider $\mathscr J_\beta+\ip{W,\cdot}$. Hence, set $\mathscr M_\beta(\mu)=\mathscr J_\beta(\mu)+\ip{W,\mu}$.
\begin{Lemma}\label{lem:Mconv}
    If $\mu_\beta$ converges weakly to $\mu$, then
    \begin{equation}
        \liminf_\beta \mathscr M_\beta(\mu_\beta)\ge \sup_{f\in C_b}\gk{\ip{f,\mu}-M(f)}\, .
    \end{equation}
\end{Lemma}
\begin{proof}
    Indeed, for every $f\in C_b$, we obtain
    \begin{equation}
        \mathscr M_\beta(\mu_\beta)\ge \ip{f,\mu_\beta}-M_\beta(f)\, .
    \end{equation}
    The convergence for $\mu_\beta$ together with Lemma~\ref{lem:FKconvergence} prove that
    \begin{equation}
        \liminf_{\beta\to\infty}\mathscr M_\beta(\mu_\beta)\ge \ip{f,\mu}-M(f)\, .
    \end{equation}
    Taking the supremum over all $f$ now yields the claim.
\end{proof}
Set $\mathscr E_\beta(\phi^2)=\mathscr M_\beta(\phi^2)+4\pi\alpha_{\mathrm{B}}\norm{\phi}_4^4$ for $\phi^2$ a probability density and extend it to all probability measures by setting it $+\infty$ for not absolutely-continuous probability measures. Recall that our goal is to compute the limit of the infimum of $\mathscr E_\beta$.
\begin{Lemma}\label{lem:lowbound}
    If $\rk{\mu_\beta}_\beta$ converges to $\mu=\phi^2\d x$, then
    \begin{equation}
        \liminf_{\beta\to\infty}\mathscr E_\beta(\mu_\beta)\ge \mathscr E(\phi^2)\, ,
    \end{equation}
    where $\mathscr E(\phi^2)=\norm{\nabla \phi}_2^2+\ip{W,\phi^2}+4\pi \alpha_{\mathrm{B}}\norm{\phi}_4^4$.
\end{Lemma}
\begin{proof}
By convex duality 
\begin{multline}
    \sup_{f\in C_b}\gk{\ip{f,\phi^2}-M(f)}\\=\begin{cases}
        \norm{\nabla \phi}_2^2+\ip{W,\phi^2}\,  ,&\textnormal{ if }\phi\in H^1(\R^d)\textnormal{ and }\norm{\phi}_2=1,\\
        +\infty &\textnormal{ otherwise.}
    \end{cases}
\end{multline}
Indeed, if we denote the right-hand side of the above by $\mathscr I_W(\mu)$, then by \eqref{eq:842026}, the $M(f)=\mathscr I_W^*(f)$. The Fenchel--Moreau theorem on the dual pair $(\Mcal(\R^d),C_b(\R^d))$ (see \cite[Chapter 1]{Zalinescu02}) then yields
\begin{equation}
    M^*(\phi^2)=\sup_{f\in C_b}\gk{\ip{f,\phi^2}-M(f)}=\mathscr I_W(\phi^2)\, ,
\end{equation}
since $\mathscr I_W$ is proper, convex and weakly lower semicontinuous by the assumptions on $W$ and properties of the standard Donsker--Varadhan rate function for Brownian occupation measures.

By Lemma~\ref{lem:Mconv}, if $\phi\in H^1(\R^d)\textnormal{ and }\norm{\phi}_2=1$, then
\begin{equation}
    \liminf_{\beta\to\infty}\mathscr M_\beta(\phi^2_\beta)\ge \norm{\nabla \phi}_2^2+\ip{W,\phi^2}\, .
\end{equation}
    The claim now follows from the dual representation
    \begin{equation}
        \int_{\R^d}\phi^4(x)\d x=\sup_{g\in C_c(\R^d)}\{2\int g(x)\phi^2(x)\d x-\norm{g}_2^2\}\, ,
    \end{equation}
    which implies liminf inequality $\liminf_{\beta\to\infty}\norm{\phi_\beta}^4\ge \norm{\phi_\beta}^4$.
\end{proof}
It remains to prove an upper bound.
\begin{Lemma}\label{lem:convlegendre}
    We have that
    \begin{equation}
        \limsup_{\beta\to\infty}\inf_{\phi^2}\mathscr E_\beta(\phi^2)\le\inf_{\phi^2} \mathscr E(\phi^2)\, .
    \end{equation}
\end{Lemma}
Lemma~\ref{lem:convlegendre} now implies Theorem~\ref{thm3}.

\noindent{\bf Proof of Theorem \ref{thm3} given Lemma~\ref{lem:convlegendre}.} 
Note that for every $C>0$, $\bigcup_{\beta>1}\gk{\phi^2\colon \mathscr E_\beta(\phi^2)\le C}$ is relatively compact. Indeed, the set is contained in $\gk{\phi^2\colon \ip{W,\phi^2}\le C}$, which is compact by Prokhorov. To show that
\begin{equation}
\lim_{\beta\to\infty}\inf_{\phi^2}\mathscr E_\beta(\phi^2)=\inf_{\phi^2}\mathscr E(\phi^2)\, ,
\end{equation}
is now standard, given the previous convergence results : to prove that the left-hand side is larger than the right, use a minimizing sequence $\phi^2_\beta$ (see e.g. Lemma~\ref{lemma-phi}). By compactness, we can extract a converging subsequence. Lemma~\ref{lem:lowbound} then concludes the proof. For the reverse direction, Lemma~\ref{lem:convlegendre} gives the result directly. \qed 

\noindent{\bf Proof of Lemma~\ref{lem:convlegendre}.} Since for every $\phi^2$
\begin{equation}
    \inf_{\phi^2}\mathscr E_\beta(\phi^2)\le \mathscr E_\beta(\phi^2)\, ,
\end{equation}
we first compute the limit of the right-hand side above, for certain nice $\phi^2$.

Fix $0\in \overline{U}\Subset\gk{W<\infty}$ connected and smooth, as well as another connected and smooth subset $0\in\overline{V}\Subset U$ and assume that the density $\phi^2$ satisfies
\begin{equation}\label{eq:652026}
    \phi\in C^2(\overline{V}),\quad \phi>0\textnormal{ in }V,\quad \phi_{\mid \partial V}\equiv 0,\quad \norm{\phi}_4<\infty\, .
\end{equation}
Assume that $\beta>1$. To first force our path to diffuse according to $\phi^2$, we tilt the measure $\P$ of the Brownian motion $(B_t)_{t\in [0,1]}$ as follows:
\begin{equation}
    \frac{\d \P\hk{[0,1]}}{\d \P}{ \mid \mathcal F_{[0,1]}}=\1\gk{H_{U^c}>1}\frac{\phi^2(B_1)}{p_U(0,B_1;1)}\, ,
\end{equation}
where $\mathcal F_{[0,1]}$ is the $\sigma$-algebra generated by paths from time $0$ to time $1$, $H_{U^c}$ is the hitting time of $U^c$ and $p_U(x,y;t)$ is the heat-kernel of the Brownian motion on $U$ with Dirichlet boundary conditions. Due to the safety distance between $U$ and $V$, $p_U(0,B_1;1)$ remains almost surely bounded from below under $\P\hk{[0,1]}$. Hence, using also that $\norm{\phi}_4<\infty$
\begin{equation}
    H\rk{\P\hk{[0,1]}|\P}=\int_V \phi^2(z)\log\frac{\phi^2(z)}{p_U(0,z;1)}\d z<\infty\, .
\end{equation}
After time $1$, we let $B_t$ evolve according to the invariant measure: let $\mathcal L\hk{\phi}$ be the generator with invariant density $\phi^2$ (recall that to align with physical conventions, our Brownian motions have generator $\Delta$ not $\Delta/2$)
\begin{equation}
    \mathcal L\hk{\phi}=\Delta+2\frac{\nabla \phi}{\phi}\cdot \nabla \, .
\end{equation}
Let $\widetilde{P}\hk{\phi}_\beta$ be the measure by first evolving $B_t$ with law $\P\hk{[0,1]}$ and then with generator $\mathcal L\hk{\phi}$. Note that $B_1$ has law $\phi^2$ under $\widetilde{P}\hk{\phi}_\beta$. We can then decompose the total entropy (see e.g. \cite[Theorem D.13]{DZ09}) and get
\begin{equation}
    H(\widetilde{P}\hk{\phi}_\beta|\P)=H\rk{\P\hk{[0,1]}|\P}+(\beta-1)\norm{\nabla\phi}_2^2\, .
\end{equation}
Hence, the specific relative entropy is given by
\begin{equation}
    \limsup_{\beta\to\infty}\frac{H(\widetilde{P}\hk{\phi}_\beta|\P)}{\beta}=\norm{\nabla\phi}_2^2\, .
\end{equation}
To define our approximating sequence, set $\mu_\beta$
\begin{equation}
    \mu_\beta=\E_{\widetilde{P}\hk{\phi}_\beta}\ek{\frac{1}{\beta}\int_0^\beta \delta_{B_s}\d s}\quad\Leftrightarrow\quad \ip{f,\mu_\beta}=\E_{\widetilde{P}\hk{\phi}_\beta}\ek{\frac{1}{\beta}\int_0^\beta f\rk{B_s}\d s}\, .
\end{equation}
By the variational characterization of the entropy, for every $f\in C_b(\R^d)$
\begin{equation}\label{eq:trnoy}
    \ip{\beta f,\mu_\beta}-\log \E\ek{e^{\int_0^\beta f(B_s)\d s}}\le  H(\widetilde{P}\hk{\phi}_\beta|\P)\, ,
\end{equation}
and hence
\begin{equation}
    \limsup_{\beta \to \infty}\mathscr J_\beta(\mu_\beta)\le \norm{\nabla \phi}_2^2\, .
\end{equation}
Since the trapping potential $W$ is bounded on $U$, we readily obtain that
\begin{equation}
    \lim_{\beta\to\infty}\ip{\mu_\beta,W}=\ip{\phi^2,W}\, .
\end{equation}
It remains to prove the convergence of the quartic GP-term. The heat-kernel of $\P\hk{[0,1]}$ is given by
\begin{equation}
    0\le p\hk{[0,1]}(0,y;s)=p_U(0,y;s)\int_V p_U(y,z;1-s)\frac{\phi^2(z)}{p_U(0,z;1)}\d z\, .
\end{equation}
By the aforementioned boundedness of $p_U(0,z;1)$ inside $V$, there must exist $C>0$ such that for all $y\in V$, $s\in [0,1]$
\begin{equation}
    p\hk{[0,1]}(0,y;s)\le Cp(0,y;s)\, , 
\end{equation}
where $p(0,y;t)$ is the free heat-kernel. Since $d=2,3$ and $U$ is bounded
\begin{equation}
    y\mapsto \int_0^1 p(0,y;s)\d s\in L^2(U)\, .
\end{equation}

Let $\phi^2_\beta$ be the density of $\mu_\beta$ with respect to the Lebesgue measure. We can write
\begin{equation}
    \phi^2_\beta(y)=\frac{1}{\beta}\int_0^1p\hk{[0,1]}(0,y;s)\d s+\frac{\beta-1}{\beta}\phi^2(y)\, .
\end{equation}
Hence
\begin{equation}
    \norm{\phi_\beta^2-\phi^2}_2\le \frac{C}{\beta}+\norm{\phi^2}_2/\beta\, ,
\end{equation}
which converges to zero as $\beta\to\infty$ by \eqref{eq:652026}. We have hence shown that for every $\phi^2$ satisfying \eqref{eq:652026}
\begin{equation}
    \limsup_{\beta \to\infty}\inf_{\phi^2}\mathscr E_\beta(\phi^2)\le \norm{\nabla \phi}_2^2+\ip{W,\phi^2}+4\pi \alpha_{\mathrm{B}}\norm{\phi}_4^4\, .
\end{equation}
General $\phi^2$ with $\mathscr{E}(\phi^2)<\infty$ can now be approximated $\phi_n^2$ satisfying \eqref{eq:652026} and we obtain
\begin{equation}
     \limsup_{\beta \to\infty}\inf_{\phi^2}\mathscr E_\beta(\phi^2)\le \inf_{\phi^2}\gk{\norm{\nabla \phi}_2^2+\ip{W,\phi^2}+4\pi \alpha_{\mathrm{B}}\norm{\phi}_4^4}\, .
\end{equation}
This concludes the proof.
\qed

\section{Proof of Theorem \ref{thm2}} 

The proof of Theorem \ref{thm2} is based on a large deviation principle for the distribution of the total intersection local time $L^{\ssup{i,j}}_\beta(0)$ of the $i^{\mathrm{th}}$ and the $j^{\mathrm{th}}$ Brownian path until time $\beta$. This object is closely related to the so-called {\it Brownian intersection measure}, formally defined as
\begin{equation}\label{int-meas}
\ell_\beta(A)= \ell_\beta^{\ssup{1,2}}(A)=\int_A \d y \prod_{i=1}^2 \int_0^\beta \d s \delta_y(B^{\ssup i}_s), \qquad A \subset \R^d,
\end{equation}
supported on the  set $S=\cap_{i=1}^2 B^{\ssup i}([0,\beta])$ which is non-empty in $d\in \{2,3\}$. Note that $L^{\ssup{i,j}}_\beta(0)= \beta^{-2}\int_{\R^d} 1 \ell_{\beta}^{\ssup{i,j}}(\d x)$ is obtained from the {\it total mass} of the locally finite measure $\beta^{-2} \ell_\beta^{\ssup{i,j}}(\cdot) \in \mathcal M(\R^d)$.

A precise definition of $\ell_\beta$ is provided by a suitable approximation scheme. Indeed let $\varphi_\eps$ be a smooth mollifier, i.e., if $\varphi$ is a non-negative, $\mathcal C^\infty$-function on $K$ compact subset of $\R^d$
with $\int_{\mathbb R^d}\varphi(y)\,\d y=1$, we take
$$
\varphi_\eps(x)=\eps^{-d}\varphi(x/\eps).
$$
Then $\int_{\R^d} \varphi_\eps=1$ and $\varphi_\eps\Rightarrow \delta_0$ weakly as probability measures. Then  the mollified occupation densities 
are defined as
\begin{equation}\label{eq:mudef}
\mu^{\ssup i}_{\eps,\beta}(y)=\varphi_\eps\star \mu^{\ssup i}_\beta(y)= \frac{1}{\beta}\int_0^\beta\d s\hspace{1mm}\varphi_{\eps}(B^{\ssup i}_s-y) \qquad i=1,2.
\end{equation}
For each fixed $\eps>0$, these are smooth bounded functions in $\R^d$, and we can take the pointwise product
$$
\ell_{\eps,\beta}\hk{1,2}(y)= \beta^2\prod_{i=1}^2 \mu_{\eps,\beta}^{\ssup i}(y)
$$
and define $\ell_{\eps,\beta}\hk{1,2}\in \Mcal(\R^d)$ to be the measure with density $\ell_{\eps,\beta}\hk{1,2}(y)$, analogously for all other indices $i,j\in\gk{1,\ldots, N}$. Note that without any mollification, the pointwise product of the occupation measures $\prod_{i=1}^2 \mu_\beta^{\ssup i}$ themselves is a discontinuous operation. 
 The smoothing procedure w.r.t. $\varphi_\eps$ alleviates the situation and makes the pointwise product $\ell_{\eps,\beta}$ a continuous functional of the smoothed occupation densities 
 $\mu_{\beta,\eps}^{\ssup i}$. $\ell_\beta(\R^d)$ is then defined as the almost-sure limit of $\ell_{\eps, \beta}(\R^d)$. See \cite{Che10} for a canonical reference on the subject.
 
The following superexponential estimate is crucial in the present context: 

\begin{Lemma}\label{lemma:superexp}
For any $a>0$,
\begin{equation}\label{eq1:superexp}
\lim_{\eps\downarrow 0} \limsup_{\beta\to\infty}
\frac 1 \beta \log  \P\rk{\beta^{-2}\abs{\ell_\beta(\R^d)-\ell_{\eps,\beta}(\R^d)}>a}
=-\infty.
\end{equation}
\end{Lemma} 
 
 Before turning to the proof of Lemma \ref{lemma:superexp}, we remark the following: when the Brownian motions are restricted to a compact subset of $\R^d$ (i.e., the free probability measure $\P$ is replaced by a sub-probability measure $\P\big(\cdot \cap_{i=1}^2 \{\tau_i>\beta\}\big)$ where $\tau_i$ is the first exit time of the $i^{\mathrm{th}}$ Borownian motion 
 from a bounded subset of $\R^d$), the estimate \eqref{eq1:superexp} was proved in \cite[Proposition 2.3]{KM11}. The proof there is based on an application of the spectral theorem of the Laplacian (with Dirichlet boundary condition) that provides a Fourier expansion of the transition sub-probability densities of the Brownian paths in terms of eigenvalues and eigenfunctions of the Laplacian. In the present context of treating free Brownian motions  in the full space $\R^d$ such an eigenvalue expansion is not available and we employ a different method, which is inspired by the argument in \cite{Muk17}; we refer to Remark \ref{rem-compare M17} below for a comparison with the argument there.

 \noindent
 We remark that if we had first tilted the measure to \textit{include} the trap $W$, the exponential tightness could have been derived using a spectral decomposition of the transition kernel (see e.g. \cite{S79} for a reference). We would also like to point out the recent results \cite{Mor23,Par26}, proving a similar concentration estimate.
 \medskip 
 
 \begin{proof}[{\bf Proof of Lemma \ref{lemma:superexp}}]

We first note that by Chebyshev's inequality, for any integer $k\in \N$, 
\begin{equation}\label{chebyshev}
\P\big[ \beta^{-2}\abs{\ell_\beta(\R^d)-\ell_{\eps,\beta}(\R^d)}>a\big] \leq a^{-k} \,\, \beta^{-2k}\,\, \E\bigg[\big|\ell_{\beta}(\R^d)-\ell_{\eps,\beta}(\R^d)\big|^k\bigg]\, .
\end{equation}
The main technical contribution is the estimation of the expected value uniformly with $k$, as we will later take $k\to\infty$ as $\beta\to\infty$. In fact, we assume from now on that $k\in 2\N$, to ignore complications from the absolute value.

By \eqref{eq:mudef}, we can rewrite
\begin{equation}
    \ell_{\eps,\beta}(\R^d)=\int_{\R^d}\mu_{\eps,\beta}\hk{1}(y)\mu_{\eps,\beta}\hk{2}(y)\d y=(2\pi)^{-d}\int_{\R^d}\widehat{\mu_{\eps,\beta}\hk{1}}(\lambda)\overline{\widehat{\mu_{\eps,\beta}\hk{2}}(\lambda)}\d \lambda\, ,
\end{equation}
where we use Parseval--Plancherel identity in the last equality and write $\widehat{\mu_{\eps,\beta}\hk{i}}$ for the Fourier transform of $\mu_{\eps,\beta}\hk{i}$, $i\in\gk{1,2}$.

Abbreviating 
\begin{equation}
    A_\beta\hk{i}(\lambda)=\int_0^\beta \ex^{i \ip{\lambda, B_u\hk{i}}}\d u\, ,
\end{equation}
we find that
\begin{equation}
    \ell_{\eps,\beta}(\R^d)=(2\pi)^{-d}\int_{\R^d} A_\beta\hk{1}(\lambda) A_\beta\hk{2}(-\lambda)\abs{\widehat{\varphi}(\eps \lambda)}^2\d \lambda\, .
\end{equation}
Due to the almost-sure continuity of the local-time process, we also obtain that
\begin{equation}
    \ell_{\beta}(\R^d)=\lim_{\eps\downarrow 0}\ell_{\eps,\beta}(\R^d)\, .
\end{equation}
We therefore introduce an additional mollification, denoted by $\eta>0$. By Fatou
\begin{equation}
    \E\bigg[\big|\ell_{\beta}(\R^d)-\ell_{\eps,\beta}(\R^d)\big|^k\bigg]\le \liminf_{\eta\downarrow 0}\E\bigg[\big|\ell_{\eta,\beta}(\R^d)-\ell_{\eps,\beta}(\R^d)\big|^k\bigg]\, .
\end{equation}
We rewrite
\begin{multline}\label{eq:intfourier}
    \E\ek{\big|\ell_{\eta,\beta}(\R^d)-\ell_{\eps,\beta}(\R^d)\big|^k}\\
    =(2\pi)^{-d k}\int_{(\R^d)^k}\prod_{r=1}^km_{\eps,\eta}(\lambda_r)G_k(\lambda_1,\ldots,\lambda_k)G_k(-\lambda_1,\ldots,-\lambda_k)\d \lambda\, ,
\end{multline}
with
\begin{equation}
    m_{\eps,\eta}(\lambda)=\abs{\widehat{\varphi}(\eta \lambda)}^2-\abs{\widehat{\varphi}(\eps \lambda)}^2\, ,
\end{equation}
and
\begin{equation}
    G_k(\lambda_1,\ldots,\lambda_k)=\int_{[0,\beta]^k}\E\ek{\ex^{i\sum_{j=1}^k\lambda_jB_{s_j}}}\d s\, .
\end{equation}
Set
\begin{equation}
    \int_{[0,\beta]\hk{k}_\leq}\d s=\int_{0\le s_1\leq\ldots\leq s_k\leq \beta }\d s_1\ldots\d s_k\, .
\end{equation}
We then use the Markov property of the Brownian motion and the different partitions of $[0,\beta]^k$ to simplify $G_k$:
\begin{equation}
    \begin{split}
         G_k(\lambda_1,\ldots,\lambda_k)&=\sum_{\sigma\in \mathfrak S_k}\int_{[0,\beta]\hk{k}_\leq}\d s\E\ek{\ex^{i\sum_{j=1}^k\lambda_{\sigma(j)}B_{s_j}}}\\
         &=\sum_{\sigma\in \mathfrak S_k}\int_{[0,\beta]\hk{k}_\leq}\d s\E\ek{\ex^{i\sum_{j=1}^k(B_{s_j}-B_{s_{j-1}})\sum_{l=j}^k\lambda_{\sigma(l)}}}\\
         &=\sum_{\sigma\in \mathfrak S_k}\int_{[0,\beta]\hk{k}_\leq}\d s\prod_{j=1}^k\ex^{-(s_j-s_{j-1})\Gamma_{j,k}(\lambda)^2}\, ,
    \end{split}
\end{equation}
with $\Gamma_{j,k}(\lambda)=|\sum_{l=j}^k\lambda_{\sigma(l)}|$ and the convention that $s_0=0$. Using the inequality
\begin{equation}
    1\le \ex^{\beta}\ex^{-\sum_{j=1}^k(s_j-s_{j-1})}\quad\textnormal{on}\quad [0,\beta]\hk{k}_\le\, ,
\end{equation}
we bound the above by
\begin{equation}
     G_k(\lambda_1,\ldots,\lambda_k)\leq \sum_{\sigma\in \mathfrak S_k}\ex^{\beta}\prod_{j=1}^k\frac{1}{1+\Gamma_{j,k}(\lambda)^2}\, .
\end{equation}
Plugging this into \eqref{eq:intfourier}, we obtain that
\begin{multline}
    \E\ek{\big|\ell_{\eta,\beta}(\R^d)-\ell_{\eps,\beta}(\R^d)\big|^k}\\
    \le C^k\ex^{2\beta}\int_{(\R^d)^k}\prod_{r=1}^k|m_{\eps,\eta}(\lambda_r)|\rk{\sum_{\sigma\in \mathfrak S_k}\prod_{j=1}^k\frac{1}{1+\Gamma_{j,k}(\lambda)^2}}^2\d \lambda\, .
\end{multline}
By the Cauchy--Schwarz inequality, for any sequence $(b_k)_k$
\begin{equation}
    \rk{\sum_{\sigma\in \mathfrak{S}_k}\prod_{i=1}^kb_{\sigma(i)}}^2\le k!\sum_{\sigma\in \mathfrak{S}_k}(\prod_{i=1}^kb_{\sigma(i)})^2\, ,
\end{equation}
and hence we can bound the integral from above by
\begin{equation}
    \sum_{\sigma\in \mathfrak S_k}\int_{(\R^d)^k}k!\prod_{r=1}^k|m_{\eps,\eta}(\lambda_r)|\rk{\prod_{j=1}^k\frac{1}{1+\Gamma_{j,k}(\lambda)^2}}^2\d \lambda\, ,
\end{equation}
with $\lambda_0=0$. By the invariance of $m_{\eps,\eta}(\lambda_r)$ under permutations, the above can be further bounded by
\begin{equation}
    \begin{split}
        (k!)^2&\int_{(\R^d)^k}\prod_{j=1}^k|m_{\eps,\eta}(\lambda_j)|\rk{1+\abs{\lambda_j+\ldots+\lambda_k}^2}^{-2}\d \lambda\\
        &=(k!)^2\int_{(\R^d)^k}\prod_{j=1}^k|m_{\eps,\eta}(\lambda_j-\lambda_{j-1})|\rk{1+\abs{\lambda_j}^2}^{-2}\d \lambda\, .
    \end{split}
\end{equation}
Since $\lambda\mapsto (1+|\lambda|^2)^{-2}\in L^1(\R^d)$ for $d=2,3$, $m_{\eps,\eta}$ bounded, and
\begin{equation}
    \liminf_{\eta\downarrow 0} m_{\eps,\eta}(\lambda)=1-\abs{\widehat{\varphi}(\eps \lambda)}^2\, ,
\end{equation}
we obtain that
\begin{equation}
    \E\bigg[\big|\ell_{\beta}(\R^d)-\ell_{\eps,\beta}(\R^d)\big|^k\bigg]\le C^k \ex^{2\beta}(k!)^2I_{\eps,k}\, ,
\end{equation}
with
\begin{equation}
    I_{\eps,k}=\int_{(\R^d)^k}\prod_{j=1}^k|m_{\eps,0}(\lambda_j-\lambda_{j-1})|w(\lambda_j)\d \lambda\, ,
\end{equation}
and $w(\lambda_j)=\rk{1+\abs{\lambda_j}^2}^{-2}$. Define the operator $K_\eps$ through its integral kernel
\begin{equation}
    K_{\eps}(x,y)=\sqrt{w(x)}m_{\eps,0}(x-y)\sqrt{w(y)}\, .
\end{equation}
As $w\in L^1(\R^d)$, note that
\begin{equation}
    \norm{K_\eps}_{\mathrm{HS}}^2=\int_{\R^d\times\R^d}w(x)|m_{\eps,0}(x-y)|^2w(y)\d x\d y< \infty\, .
\end{equation}
By dominated convergence (as $\abs{m_{\eps,0}(x-y)}\le 2$ and converges to zero as $\eps\to 0$)
\begin{equation}
    \lim_{\eps\to 0}  \norm{K_\eps}_{\mathrm{HS}}=0\, .
\end{equation}
Furthermore, we have that
\begin{equation}
    I_{\eps,k}=\ip{\sqrt{w}, K_\eps^{k-1}(\sqrt{w}m_{\eps,0})}\le \norm{\sqrt{w}}_2\norm{K_\eps}^{k-1}_{\mathrm{op}}\norm{\sqrt{w}m_{\eps,0}}_2\le C \norm{K_\eps}_{\mathrm{HS}}^{k-1}\, .
\end{equation}
This finally gives that
\begin{equation}
     \E\ek{\big|\ell_{\beta}(\R^d)-\ell_{\eps,\beta}(\R^d)\big|^k}\le C^k \ex^{2\beta}(k!)^2\norm{K_\eps}_{\mathrm{HS}}^{k-1}\, .
\end{equation}
Choosing $k\in 2\N$ with $k=\beta(1+o(1))$ yields then with \eqref{chebyshev} and Stirling's formula that
\begin{equation}
    \limsup_{\beta\to\infty}
\frac 1 \beta \log  \P\rk{\beta^{-2}\abs{\ell_\beta(\R^d)-\ell_{\eps,\beta}(\R^d)}>a}\le C+
\log(C\norm{K_\eps}_{\mathrm{HS}}/a).
\end{equation}
Hence, for $a>0$ fixed
\begin{equation}
   \lim_{\eps\downarrow 0}  \limsup_{\beta\to\infty}
\frac 1 \beta \log  \P\rk{\beta^{-2}\abs{\ell_\beta(\R^d)-\ell_{\eps,\beta}(\R^d)}>a}=-\infty\, .
\end{equation}
 \end{proof}
 \begin{Rem}\label{rem-compare M17}
     As mentioned above, the proof of Lemma \ref{lemma:superexp} via Fourier transform is inspired by the proof of \cite[Lemma 3.4]{Muk17}. The use of Fourier transforms to study Brownian intersection local times dates back to the seminal work of Geman, Horowitz, and Rosen \cite{GHR84} and has since been developed  (e.g. \cite{Ros83}). The proof presented here however differs from that in \cite{Muk17}, where an additional exponential smoothing plays a crucial role, whereas no such smoothing is used in the present argument. Moreover, other technical details of the proof are quite different from those in \cite{Muk17}. We also take this opportunity to correct the statement of \cite[Lemma 3.4]{Muk17} and clarify a couple of points in its proof. First, the statement of \cite[Lemma 3.4]{Muk17} should be stated, just as our Lemma \ref{lemma:superexp} above, as follows: 
     
     \noindent{\bf Lemma} [\cite[Lemma 3.4]{Muk17}:] For any $a>0$, 
$$
\limsup_{\varepsilon\to 0} \limsup_{t\to\infty} \frac 1t \log \mathbb P^{\otimes}\bigg( t^{-2} |\ell_t(\mathbb R^3) - \ell_{\varepsilon,t}(\mathbb R^3)| \geq a\bigg)=-\infty. 
$$
(Here $\ell_t(\mathbb R^3)$ is the total mass of the intersection measure on $\mathbb R^3$, or the intersection local time, $t^{-2} \ell_{\varepsilon,t}(y)= L_{\varepsilon,t}^{(1)} (y) L_{\varepsilon,t}^{(1)} (y)$ and $t^{-2} \ell_{\varepsilon,t}(\mathbb R^3)= \int_{\mathbb R^3} t^{-2} \ell_{\varepsilon,t}(y) \mathrm d y$. There should not be any $f$ appearing in the statement nor in the proof of \cite[Lemma 3.4]{Muk17}.\footnote{The second author would like to thank R. Fukushima and T. Mori who correctly pointed out to him that the first identity on page 2390 of \cite{Muk17} is valid only for the total masses $\ell_{\eps,t}(\R^3)$ and $\ell_t(\R^3)$ and not for the versions with a test function $f$ as incorrectly stated there.} The above statement with the total masses of the intersection local times suffices entirely for the purposes of that paper in Section 3.)

\medskip

\noindent{\bf Proof.} We follow the proof of Lemma 3.4 of \cite{Muk17} exactly and make slight modifications. We replace the single exponential time $\tau$ by two independent exponential times 
\[
\tau_1,\tau_2 \stackrel{\text{i.i.d.}}{\sim}\mathrm{Exp}(1),
\]
independent of $W^{(1)}, W^{(2)}$. Then the exact repetition of the proof of Eq. (3.16) of \cite{Muk17}  gives 
\begin{equation}\label{32}
\lim_{\eps\downarrow 0} \limsup_{k\uparrow\infty} \frac 1k \log \bigg[\frac 1 {k!^2} \,\,\big[\E\otimes E^{\tau_1\times\tau_2}\big]\bigg[\bigg(\ell_{\tau_1,\tau_2}(\R^3)- \ell_{\eps,\tau_1,\tau_2}(\R^3)\bigg)^k\bigg]\bigg]=-\infty,
\end{equation} 
where $\ell_{\tau_1,\tau_2}$ is the intersection local time defined for the first path running until time $\tau_1$ and the second path running until time $\tau_2$, 
and the same holds for the mollified version $\ell_{\varepsilon;\tau_1,\tau_2}$ as well.

\medskip 

The exact identities on page 2390 and 2391 of \cite{Muk17} imply that if we define  
$$
\hat\tau:= \min\{\tau_1,\tau_2\} \sim \mathrm{Exp}(2),
$$
then since the integrands on these identities are non-negative (since $1- \widehat\varphi_\eps(\cdot)\geq 0$ which is seen by taking a Gaussian mollifier as in the paper), we have 
\begin{equation}\label{domination}
\E{\otimes}  E^{\hat\tau}\bigg[\bigg(\ell_{\hat\tau}(\R^3)- \ell_{\eps,\hat\tau}(\R^3)\bigg)^k\bigg] \leq \E\otimes E^{\tau_1\otimes\tau_2} \bigg[\bigg(\ell_{\tau_1,\tau_2}(\R^3)- \ell_{\eps,\tau_1,\tau_2}(\R^3)\bigg)^k\bigg].
\end{equation} 
\medskip
Combining \eqref{domination} and \eqref{32} imply 
\begin{equation}\label{32hat}
\lim_{\eps\downarrow 0} \limsup_{k\uparrow\infty} \frac 1k \log \bigg[\frac 1 {k!^2} \,\,\big[\E\otimes E^{\hat\tau}\big]\bigg[\bigg(\ell_{\hat\tau}(\R^3)- \ell_{\eps,\hat\tau}(\R^3)\bigg)^k\bigg]\bigg]=-\infty.
\end{equation} 
Now use scaling for $\tau:=2\hat\tau\sim\mathrm{Exp}(1)$ as in page 2389 of \cite{Muk17}  to get 
$$
\begin{aligned}
\big[\E\otimes E^{\tau}\big]\bigg\{\ell_\tau(\R^3)-\ell_{\eps,\tau}(\R^3)\rangle^k\bigg\}&= E^{\tau}\big[(\tau t^{-1})^{k/2}\big] \, \,\E\bigg\{\big| \ell_t(\R^3)- \ell_{\eps,t}(\R^3)\big\rangle\big|^k\bigg\}\\
&= t^{-k/2} \, \Gamma\bigg(1+\frac k 2\bigg) \,\,\E\bigg\{|\ell_t(\R^3)- \ell_{\eps,t}(\R^3)\big|^k\bigg\}.
\end{aligned}
$$
Use now Chebysev's inequality as in Eq. (3.15) of \cite{Muk17}  and combine with \eqref{32hat} to get the statement of the lemma above. This completes the proof.

     \qed
 \end{Rem}

  Having exponential approximation of the local time by its mollified version, the proof of Theorem~\ref{thm2} is rather straight-forward: the approximation allows us to replace the true Hamiltonian $\mathscr L_{N,\beta}$ by a well-behaved variant $ \mathscr L_{N,\beta,\eps}$. Using Varadhan's lemma together with the weak-LDP of the occupation measures yield an $\eps$-dependent variational principle. The trap allows us to remove the mollification parameter.
  
 \medskip
 
 \begin{proof}[{\bf Proof of Theorem \ref{thm2}}]
 \noindent
 Our goal is to calculate $\E\ek{\ex^{-\mathscr L_{N,\beta}}}$. Using the previously made rigorous approximation $\beta^{-2}\ell_{\eps,\beta}\hk{i,j}(\R^d)\approx L\hk{i,j}(0)$ by Lemma~\ref{lemma:superexp}, define the approximated version
 \begin{equation}
     \mathscr L_{N,\beta,\eps}=\beta   \sum_{1\leq i < j \leq N}\beta^{-2}\ell_{\eps,\beta}\hk{i,j}(\R^d)+
\sum_{i=1}^N \int_0^\beta W(B^{\ssup i}_s)\d s. 
 \end{equation}
Since $N$ is fixed (for now), Lemma~\ref{lemma:superexp} and a union bound implies that for every $a>0$
\begin{equation}\label{eq:expconcentration}
    \lim_{\eps\downarrow 0}\limsup_{\beta\to \infty}\frac{1}{\beta}\log \P\rk{\exists 1\le i<j\le N\colon \abs{L\hk{i,j}(0)-\beta^{-2}\ell_{\eps,\beta}\hk{i,j}(\R^d)}>a}=-\infty\, ,
\end{equation}
we obtain that
\begin{equation}
     \lim_{\eps\downarrow 0}\limsup_{\beta\to \infty}\frac{1}{\beta}\abs{\log\E\ek{\ex^{-\mathscr L_{N,\beta}}}-\log \E\ek{\ex^{-\mathscr L_{N,\beta,\eps}}}}=0\, .
\end{equation}
Indeed, note first that
\begin{equation}
    \lim_{\eps\downarrow 0}\limsup_{\beta\to \infty}\frac{1}{\beta}\log\E\ek{\ex^{-\mathscr L_{N,\beta}}}>-\infty\, ,
\end{equation}
and the same when replacing $\mathscr L_{N,\beta}$ by $\mathscr L_{N,\beta,\eps}$, using the entropy inequality
\begin{equation}
    \log\E\ek{\ex^{-\mathscr L_{N,\beta}}}\ge -\E_P[\mathscr L_{N,\beta}]-H(P\mid \P)\, ,
\end{equation}
and the constructed trial states $P=\bigotimes_{i=1}^N\widetilde{P}\hk{\phi}_\beta$ from the proof of Lemma~\ref{lem:convlegendre}. 

Set $G=\{\max_{i<j}\abs{L\hk{i,j}(0)-\beta^{-2}\ell_{\eps,\beta}\hk{i,j}(\R^d)}\le a\}$, then on the event $G$, the difference $\abs{\mathscr L_{N,\beta,\eps}-\mathscr L_{N,\beta}}$ is bounded by $\beta \binom{N}{2}a$, which is negligible under the $\beta^{-1}\log$ scaling by letting $a\downarrow 0$ first. On $G^c$, use the positivity of the interaction combined the superexponential concentration \eqref{eq:expconcentration}.

We now want to write $\mathscr L_{N,\beta,\eps}$ as a functional of the occupation measures $\rk{\mu_i}_i$. This was already done in \eqref{eq:mudef}, however, we keep now make the mapping $\mu\mapsto\mu_{\eps}$ explicit: for two probability measures $\mu,\nu$, set
\begin{equation}
    F_\eps(\mu,\nu)=\int_{\R^d}\varphi_\eps\star \mu (x)\varphi_\eps\star\nu(x)\d x\, ,
\end{equation}
and set
\begin{equation}
    \Phi_\eps(\mu_\beta\hk{1},\ldots,\mu_\beta\hk{N})=\sum_{1\le i<j\le N}F_\eps(\mu_\beta\hk{i},\mu_\beta\hk{j})+\sum_{i=1}^N\ip{W,\mu_\beta\hk{i}}\, .
\end{equation}
Note that $F_\eps$ is bounded and continuous, while $\Phi_\eps$ is not, due to the trapping part. It is however non-negative and lower-semicontinuous. 

We first prove
\begin{equation}
    \lim_{\beta\to \infty}\frac{1}{\beta}\log \E\ek{\ex^{-\mathscr L_{N,\beta,\eps}}}=-\inf_{\mu\in \mathcal{M}_1(\R^{d})^{ N}}\gk{I_N(\mu)+\Phi_\eps(\mu)}\, ,
\end{equation}
where
\begin{multline}
    I_N(\mu)=I_N(\mu\hk{1},\ldots,\mu\hk{N})\\
    =\begin{cases}
        \sum_{i=1}^N\norm{\nabla\phi_i}_2^2&\textnormal{ if }\forall i\colon \frac{d\mu\hk{i}}{\d x}=\phi_i(x)^2\textnormal{ with }\phi_i\in H^{1}\, ,\\
        +\infty&\textnormal{ otherwise.}
    \end{cases}
\end{multline}
If $(\mu\hk{1},\ldots,\mu\hk{N})$ was satisfying a strong large deviation principle, the result would be more or less immediate, from Varadhan's lemma. However, the vector only satisfies a weak LDP. The coercivity of the trap allows us to conclude nonetheless.

\noindent
\textbf{Upper bound}: we compactify the problem. Set
\begin{equation}
    K_R=\gk{\mu=(\mu\hk{1},\ldots,\mu\hk{N})\colon \Phi_\eps(\mu)\le R}\, .
\end{equation}
Since $\Phi_\eps(\mu)\ge \sum_{i=1}^N\ip{W,\mu_i}$, $K_R$ is compact. We then have
\begin{equation}
    \E\ek{\ex^{-\mathscr L_{N,\beta,\eps}}}\le \E\ek{\ex^{-\mathscr L_{N,\beta,\eps}}, K_R}+\ex^{-\beta R}\, .
\end{equation}
Varadhan's lemma gives
\begin{equation}
     \limsup_{\beta\to \infty}\frac{1}{\beta}\log \E\ek{\ex^{-\mathscr L_{N,\beta,\eps}}, K_R}\le -\inf_{\mu\in  K_R}\gk{I_N(\mu)+\Phi_\eps(\mu)}\, .
\end{equation}
This implies
\begin{equation}
     \limsup_{\beta\to \infty}\frac{1}{\beta}\log \E\ek{\ex^{-\mathscr L_{N,\beta,\eps}}}\le -\lim_{R\to\infty}\min\{\inf_{\mu\in  K_R}\gk{I_N(\mu)+\Phi_\eps(\mu)},R\}\, .
\end{equation}
Monotonicity then implies the claim.

\noindent
\textbf{Lower bound}: this is similar to the previous section. The entropy inequality from the variational characterization gives that for every $\nu_\beta$ probability measure on the $N$-path space
\begin{equation}
    \frac{1}{\beta}\log \E\ek{\ex^{-\mathscr L_{N,\beta,\eps}}}\ge -\frac{1}{\beta}\E_{\nu_\beta}\ek{\mathscr L_{N,\beta,\eps}}-\frac{1}{\beta}H(\nu_\beta|\P^{\otimes N})\, .
\end{equation}
For $\mu=(\mu_1,\ldots,\mu_N)$, assume that $\phi^2_i=\frac{\d \mu\hk{i}}{\d x}$ is compactly supported inside $\gk{W<\infty}$ and satisfies \eqref{eq:652026}. Then choose as in the proof of Lemma~\ref{lem:convlegendre}
\begin{equation}
    \nu_\beta=\widetilde{P}\hk{\phi_1}_\beta\otimes\cdots\otimes\widetilde{P}\hk{\phi_N}_\beta\, .
\end{equation}
The entropy term converges as in Lemma~\ref{lem:convlegendre}
\begin{equation}
    \lim_{\beta\to\infty }\frac{1}{\beta}H(\nu_\beta|\P^{\otimes N})=I_N(\mu)\, .
\end{equation}
Similar, using the boundedness of the mollified kernel ($\eps$ is still fixed!)
\begin{equation}
    \lim_{\beta\to\infty }\frac{1}{\beta}\E_{\nu_\beta}\ek{\mathscr L_{N,\beta,\eps}}=\Phi_\eps(\mu)\, .
\end{equation}
The general case then follows via approximation and taking the infimum over all $\mu$.

\noindent
\textbf{Removing the $\eps$-cut-off}: it remains to show that
\begin{equation}
    \lim_{\eps\to 0}\inf_{\mu\in \mathcal{M}_1(\R^{d})^{\otimes N}}\gk{I_N(\mu)+\Phi_\eps(\mu)}=\inf_{\mu\in \mathcal{M}_1(\R^{d})^{\otimes N}}\gk{I_N(\mu)+\Phi(\mu)}\, ,
\end{equation}
with
\begin{equation}
    \Phi(\mu)=\sum_{1\le i<j\le N}\ip{\frac{\d \mu\hk{i}}{\d x},\frac{\d \mu\hk{j}}{\d x}}+\sum_{i=1}^N\ip{W,\mu\hk{i}}\, .
\end{equation}
The upper bound is quick: by the Sobolev embedding, for $\mu=(\phi_1,\ldots,\phi_N)$, one has $\phi_i\in L^4(\R^d)$. This gives that $\lim_{\eps\downarrow 0}F_\eps(\phi_i^2,\phi_j^2)=\ip{\phi_i^2,\phi_j^2}$.

For the lower bound, the argument is equivalent to that of Theorem~\ref{thm3}: fix any sequence of approximate minimizers $\rk{\mu_\eps}_\eps$ with $\mu_\eps=(\phi\hk{\eps}_i)^2\d x$. Similar to before, the trap $W$ gives tightness in $L^2$. The fact that the $\rk{\mu_\eps}_\eps$ are approximate minimizers gives the boundedness of $\norm{\nabla \phi_i\hk{\eps}}_2$, uniformly in $\eps\downarrow 0$. Hence, $(\phi\hk{\eps}_i)^2_\eps$ converge strongly in $L^2$, using the Sobolev embeddings. Denote the respective limits as $(\phi_i)^2$. Finally, note that by Young’s convolution inequality
\begin{equation}
    \norm{\varphi_\eps\star(\phi\hk{\eps}_i)^2-\phi_i^2 }_2\le \norm{(\phi\hk{\eps}_i)^2-\phi_i^2}_2+\norm{\varphi_\eps\star\phi_i^2-\phi_i^2}_2\, ,
\end{equation}
which vanishes as $\eps\downarrow 0$. Hence, it follows that we remove the $\eps$-cut-off.

  \end{proof}

\noindent{\bf Acknowledgements: } The research of the second author is funded by the Deutsche Forschungsgemeinschaft (DFG) under Germany's Excellence Strategy EXC 2044--390685587, Mathematics M\"unster: Dynamics--Geometry--Structure.

\end{document}